\documentclass[10pt,a4paper]{amsart}
\usepackage[english]{babel}
\usepackage[all]{xy}
\usepackage{amsthm}
\usepackage{amssymb,xypic,amscd,graphics,amsmath}
\CompileMatrices

\newtheorem{defn}{Definition}
\newtheorem{thm}[defn]{Theorem}
\newtheorem{cor}[defn]{Corollary}
\newtheorem{lem}[defn]{Lemma}

\theoremstyle{remark}
\newtheorem{rem}[defn]{Remark}
\theoremstyle{remark}
\newtheorem{exam}{Example}
\numberwithin{equation}{section} \numberwithin{defn}{section}

\newcommand\aut{\operatorname{Aut}}

\newcommand\gl{\operatorname{Gl}}

\newcommand\id{\operatorname{Id}}
\renewcommand\det{\operatorname{det}}

\newcommand\ed{\operatorname{End}}

\newcommand\Det{\operatorname{det}}
\newcommand\tr{\operatorname{Tr}}

\newcommand\res{\operatorname{Res}}

\renewcommand\deg{\operatorname{deg}}

\renewcommand\tilde{\widetilde}

\renewcommand\O{{\mathcal O}}

\begin{document}

\title[General Reciprocity Law]{A General Reciprocity Law for Symbols \\ on Arbitrary Vector Spaces}
\author{Fernando Pablos Romo}
\address{Departamento de
Matem\'aticas and Instituto Universitario de F\'{\i}sica Fundamental y Matem\'aticas, Universidad de Salamanca, Plaza de la Merced 1-4,
37008 Salamanca, Espa\~na} \email{fpablos@usal.es}
\thanks{
\\ (*) Corresponding author. \\ This work is partially supported by the
Spanish Government research contract no. MTM2012-32342.}

\begin{abstract} The aim of this work is to offer a general theory of reciprocity laws for symbols
on arbitrary vector spaces, and to show that classical explicit reciprocity laws
are particular cases of this theory (sum of valuations on a complete curve, Residue Theorem, Weil Reciprocity Law and the Reciprocity Law for the Hilbert Norm Residue Symbol). Moreover,
several reciprocity laws introduced over the past few years by D. V. Osipov, A. N. Parshin, I. Horozov, I. Horozov - M. Kerr and the author -together with D. Hern\'andez Serrano-, can also be deduced from this general expression.
\end{abstract}
\maketitle

\tableofcontents
\setcounter{tocdepth}1

\section{Introduction}

      In 1968 J. Tate \cite{Ta} gave a definition
of the residues of differentials on curves in terms of traces of
certain linear operators on infinite-dimensional vector spaces.
Furthermore, he proved the residue theorem (the additive
reciprocity law) from the finiteness of the cohomology groups
$H^0(C,{\mathcal O}_C)$ and $H^1(C,{\mathcal O}_C)$.

    A few years later, in 1971, J. Milnor \cite{Mi} defined the
tame symbol $(\cdot, \cdot)_v$ associated with a discrete
valuation $v$ on a field $F$. Explicitly, if $A_v$ is the
valuation ring, $p_v$ is the unique maximal ideal and $k_v =
{A_v}/{p_v}$ is the residue class field,
 Milnor defined $(\cdot, \cdot)_v \colon F^\times\times F^\times \to k_v^\times$
by $$(f, g)_v) = (-1)^{v(f)\cdot v(g)} \frac {f^{v(g)}}{g^{v(f)}}
(\text {mod }p_v).$$

(Here and below $R^\times$ denotes the multiplicative group of a
ring $R$ with unit.)

    This definition generalizes the definition of the multiplicative
local symbol given by J. P. Serre in \cite{Se}. If $C$ is a
complete, irreducible and non-singular curve over an algebraically
closed field $k$, and $\Sigma_C$ is its field of functions, the
expression
$$\prod_{x\in C} (-1)^{v_x(f)\cdot v_x(g)} \frac
{f^{v_x(g)}}{g^{v_x(f)}} (x) = 1$$\noindent for all $f,g \in
\Sigma_C^\times$ is the Weil Reciprocity Law.

  In 1989, Arbarello, De Concini and
V.G. Kac -\cite{ACK}- provided a new definition of the tame symbol
of an algebraic curve from the commutator of a certain central
extension of groups and, analogously to Tate's construction, they
deduced the Weil Reciprocity Law from the finiteness of its
cohomology. More recently, in \cite{Pa2} the author has obtained a
generalization of this result for the case of a complete curve
over a perfect field. In both cases, similar to Tate's proof of
the Residue Theorem, the reciprocity laws were proved by using
``ad hoc'' arguments from the properties of complete algebraic
curves.

    Furthermore, using similar arguments, the reciprocity law for the Hilbert Norm Residue Symbol (for definition see \cite{Sc}) was
deduced by the author in \cite{Pa7}.

    In all of these cases, there exists a set $X$ of subspaces of a $k$-vector space $V$, a group $G$, and
a symbol $f$, that is a map $f\colon X \longrightarrow G$, satisfying the conditions that $f$ is independent of commensurability and $$f(A)\cdot f(B) = f(A+B) \cdot f(A\cap B)\, ,$$ for all $A,B \in X$.

    The goal of the present work is to offer a general reciprocity law
for symbols $f\colon X \longrightarrow G$, such that several reciprocity laws will be
particular cases of this general theory. Indeed, classical explicit reciprocity laws
are particular cases of this theory (sum of valuations on a complete curve, Residue Theorem, Weil Reciprocity Law and the Reciprocity Law for the Hilbert Norm Residue Symbol). Moreover,
several reciprocity laws introduced over the past few years by D. V. Osipov (\cite{Os}) , A. N. Parshin (\cite{Par}), I. Horozov (\cite{Ho}), I. Horozov - M. Kerr (\cite{HoK}) and the author -together with D. Hern\'andez Serrano- (\cite{HP}), can also be deduced from this general expression.

    The organization of the paper is as follows. Section \ref{s:prel}
contains a brief summary of the results on the commutators and the index
appearing in \cite{ACK} and \cite{Pa2} (Subsection \ref{ss:tame-esther}), on the definition and properties of the abstract
residue introduced in \cite{Ta} (Subsection \ref{ss:abst-resi}), and an approach to the Segal-Wilson pairing according to the statements of \cite{HP} (Subsection \ref{ss:Segal-Sil-Esther}).

Finally, Section
\ref{s:explicit} is devoted to the main results of this work:
to prove the General Reciprocity Law (Theorem \ref{thm:GRLs}),
and, as an application, to show that explicit reciprocity laws
can be deduced from the general expression. As particular cases of the statement of this Theorem we
obtain proofs for the sum of valuations of a function on a complete curve (Subsection \ref{sss:sum-valuations}); for a reciprocity
law of the symbol $\nu_{x,C}$ on a surface studied by D. V. Osipov (Subsection \ref{ss:sum-nu-surface}); for the Weil Reciprocity Law (Subsection \ref{sss:curve}); for the reciprocity
law of the Hilbert Norm Residue Symbol (Subsection \ref{sss:Hilbert-norm-reciprocity}); for reciprocity laws of the refinament of the Parshin symbol given recently by Horozov, and of the Parshin symbol on a surface (Subsection \ref{ss:Horozov-Parshin}); for the reciprocity law of the Horozov-Kerr 4-function local symbol (Subsection \ref{ss:Horozov-Kerr}); for the Residue Theorem (Subsection \ref{sss:residue-theorem-proof}), and for the reciprocity law of the Segal-Wilson pairing (Subsection \ref{ss:Segal-Wilson-reciprocity}).

It should be noted that the General Reciprocity Law (Theorem \ref{thm:GRLs}) is valid for each vector space over an arbitrary
field and, therefore, it should be possible to study other
reciprocity laws of symbols or to deduce new explicit expressions
in Algebraic Number Theory using this method.

\section{Preliminaries} \label{s:prel}

\subsection{Tame symbol on $k$-vector spaces.}\label{ss:tame-esther}
\label{ss:commensurability}

        Let $k$ be an arbitrary field, let $V$ be a $k$-vector space and let $A, B$ be two
$k$-vector subspaces of $V$.

\begin{defn} \label{d:comm0} $A$ and $B$ are said to be
commensurable if $A + B /A\cap B$ is a finite dimensional vector
space over $k$. We write $A\sim B$ to denote commensurable
subspaces -\cite{Ta}-.
\end{defn}

     In 1989, fixing a vector subspace, $V_+ \subset V$,
if $k$ is algebraically closed  and $$\gl (V,V_+) = \{f\in
{\aut}_k(V) \text { such that } f(V_+)\sim V_+\}\, ,$$\noindent E. Arbarello, C. de Concini and V.-G. Kac \cite{ACK}
constructed a determinantal central extension of groups:
\begin{equation}\label{e:greg} 1\to k^\times \longrightarrow {\widetilde {\gl (V,V_+)}}
\overset {\pi} \longrightarrow \gl (V,V_+) \to 1\,
,\end{equation}\noindent and they used this extension to study the
tame symbol on an algebraic curve.

    In 2002, the author of this work generalized the results of \cite{ACK}
to the case of vector spaces over an arbitrary ground field
-\cite{Pa2}-.

    Given an element $f\in \gl (V,V_+)$, the ``index of $f$ over $V_+$'' is the integer
number:
$$i(f,V_+,V) = \text {dim}_k({V_+}/{V_+\cap fV_+}) - \text {dim}_k(
{fV_+}/{V_+\cap fV_+})\, .$$

If $\text{dim}_k (V_+) < \infty$, it is clear that $i(f,V_+,V) = 0$.

    We denote by $\{\cdot, \cdot\}^{V}_{V^+}$ the
 commutator of the central extension (\ref{e:greg}); that is, if $\sigma$ and $\tau$ are two
commuting elements of $\gl (V,V_+)$ and ${\tilde \sigma}, {\tilde
\tau}\in {\widetilde {\gl (V,V_+)}}$ are elements such that $\pi
({\tilde \sigma}) = \sigma$ and $\pi ({\tilde \tau}) = \tau$, then
one has a commutator pairing:
$$\{\sigma,\tau\}_{V}^{V_+} = {\tilde \sigma}\cdot {\tilde \tau}\cdot {\tilde \sigma}^{-1} \cdot {\tilde
\tau}^{-1} \in k^\times\, .$$

\begin{defn} \label{def:tame-vector} For an arbitrary vector subspace $V_+ \subseteq V$, we shall use the term ``tame symbol
associated with $V_+$'' to refer to the expression
$$<\sigma,\tau>_{V_+}^V = (-1)^{i(\sigma,V_+,V)\cdot i(\tau,V_+,V)}\{\sigma,\tau\}_{V_+}^V \in k^{\times}\, ,$$\noindent for all commuting elements $\sigma, \tau \in \gl (V,V_+)$.
\end{defn}

\begin{rem} \label{rem:curve} Let $C$ be an irreducible, non-singular and complete curve
over a perfect field $k$, and let $\Sigma_C$ be its function field.
Similar to the above, each closed point $x\in C$ corresponds to a discrete valuation ring
$\O_x$ with field of fractions $K = \Sigma_C$ that contains $k$.
Write $A_x = \hat {\O_x}$, the completion of $\O_x$ and write
$K_x$ for the field of fractions of $A_x$ (which is the completion
of $\Sigma_C$ with respect to the valuation defined by $\O_x$).
With these notations, it
follows from \cite{Pa2} (Section 5) that, similar to \cite{ACK},
we have a central extension of groups
$$1 \to k^\times \to {\tilde {\gl}}(K_x, A_x) \to \gl (K_x, A_x) \to
1\, ,$$\noindent which, since $\Sigma_C^*\subseteq \gl (K_x,
A_x)$, induces by restriction another determinantal central
extension of groups:
$$1 \to k^\times \to {\tilde {{\Sigma}_C^\times}} \to {\Sigma}_C^\times\to 1\,
,$$\noindent whose commutator, for all ${f},{g} \in
{\Sigma}_C^\times$, is:
$$\{{f}, {g}\}_{A_x}^{K_x} =
N_{k(x)/k}[\frac{{f}^{v_x({g})}}{{g}^{v_x({f})}}(x)] \in
k^\times\, ,$$\noindent where $k(x)$ is the residue class field of
the closed point $x$ and $N_{k(x)/k}$ is the norm of the extension
$k\hookrightarrow k(x)$.

    Hence, if $\deg(x) = \text{dim}_k k(x)$, according to \cite{Pa2} the tame symbol associated with a closed
point $x\in C$ is the map
$$<\cdot, \cdot>_x\colon {\Sigma}_C^\times \times
{\Sigma}_C^\times \longrightarrow k^\times\, ,$$ \noindent defined
by:
$$<f,g>_x = (-1)^{\deg(x)\cdot v_x(f)\cdot v_x(g)}\cdot
N_{k(x)/k}[\frac{{f}^{v_x({g})}}{{g}^{v_x({f})}}(p)]$$\noindent
for all  $f,g \in {\Sigma}_C^\times$.

    When $x$ is a rational point of $C$, this definition coincides
with the tame symbol associated with the field ${\Sigma}_C$ and
the discrete valuation $v_x$ (the multiplicative local symbol
-\cite{Se}-).
\end{rem}

\subsubsection{Properties of $\{\cdot,\cdot\}^{V}_{V_+}$ and $<\cdot,\cdot>^{V}_{V_+}$.} \label{sss:reaezd}

    Let us set elements $\sigma,\sigma',\tau,\tau'\in \gl (V,V_+)$
such that the $\sigma$'s commute with the $\tau$'s. (But we need
assume neither that $\sigma\sigma'=\sigma'\sigma$ nor that
$\tau\tau'=\tau'\tau$.) Hence, according to the statements of
 \cite{ACK} and \cite{Pa2}, the
following relations hold:\\

\begin{enumerate}
\item $\{\sigma,\sigma\}^{V}_{V_+}=1$.\\
\item $\{\sigma,\tau\}^{V}_{V_+}=\left(\{\tau,\sigma\}^{V}_{V_+}\right)^{-1}$.\\

\item $\{\sigma\sigma',\tau\}^{V}_{V_+}
=\{\sigma,\tau\}^{V}_{V_+}\cdot \{\sigma',\tau\}^{V}_{V_+}$.\\

\item $\{\sigma,\tau\tau'\}^{V}_{V_+}
=\{\sigma,\tau\}^{V}_{V_+}\cdot \{\sigma,\tau'\}^{V}_{V_+}$.\\

\item If  $\sigma V_{+} = V_{+} = \tau V_{+}$, then we have:
\begin{equation} \label{eqprop:equal}
\{\sigma,\tau\}^{V}_{V_+}= 1 \, . \end{equation}

\item If $V_{+}=\{0\}$ or
$V_{+}=V$, then $\{\sigma,\tau\}^{V}_{V_+}=1$ and
\begin{equation} \label{eprop:rqerq} <\sigma,\tau>^{V}_{V_+}=1\, .\end{equation}

\item If $V_{+} \subseteq {\tilde V} \subseteq V$ and $\sigma
{\tilde V}  = {\tilde V} = \tau {\tilde V}$, then $\{\sigma,\tau\}^{V}_{V_+}= \{\sigma,\tau\}^{{\tilde V}}_{V_+}$ and
\begin{equation} \label{eqprop:recrt-equal}
<\sigma,\tau>^{V}_{V_+}= <\sigma,\tau>^{{\tilde V}}_{V_+} \, .
\end{equation}

\item $\{\sigma,\tau\}^{V}_{V_+}$ depends only on the
commensurability class of $V_{+}$.\\

\item If $A$ and $B$ are two $k$-vector subspaces of $V$, for
all commuting elements $\sigma, \tau \in \gl({V},A)\cap
\gl({V},B)$ one has that $\sigma, \tau \in \gl({V},A+B)\cap
\gl({V},A\cap B)$ and:
\begin{equation} \label{eq:moni} \{\sigma, \tau\}_{A}^{V}\cdot \{\sigma, \tau\}_{B}^{V}
= (-1)^{\beta}\cdot
 \{\sigma, \tau\}_{A+B}^{V} \cdot \{\sigma, \tau\}_{A\cap
B}^{V}\, ,\end{equation}\noindent where $$\begin{aligned}\beta =
i(\tau,A,V)\cdot i(\sigma,B,V) +
 i(\tau,A,V)&\cdot i(\sigma,B,V)  \\ + i(\tau,A+B,V)\cdot i(\sigma,A\cap B,V) +
 i(\tau,A\cap B,&V)\cdot i(\sigma,A+B,V)\,
.\end{aligned}$$\\

\item If $A$ and $B$ are two $k$-vector subspaces of $V$, for
all commuting elements $\sigma, \tau \in \gl({V},A)\cap
\gl({V},B)$ one has that \begin{equation} \label{eq:moni-izqe}{\bf  <\sigma,\tau>_A^V \cdot <\sigma,\tau>_B^V = <\sigma,\tau>_{A+B}^V \cdot <\sigma,\tau>_{A\cap B}^V\, .}\end{equation}

\end{enumerate}

Also, from expression (\ref{eq:moni}) it may be deduced that:

\begin{lem} \label{lem:izq-mon} If $V = V_+ + {\tilde V}$ and $\tau, \sigma \in \gl
(V,V_+)$ with $\sigma ({\tilde V}) = \tau ({\tilde V}) = {\tilde
V}$, then $\tau, \sigma \in \gl ({\tilde V},V_+ \cap {\tilde V})$
and $$\{\tau, \sigma\}_{V_+}^{V} = \{\tau, \sigma\}_{V_+\cap
{\tilde V}}^{{\tilde V}}\, .$$
\end{lem}

\subsubsection{Properties of the index $i(f,V_+,V)$.}

    One has that:
\begin{enumerate}
 \item   If $A$ and $B$ are two vector subspaces of $V$, and $\sigma \in \gl({V},A)\cap
\gl({V},B)$, a straightforward computation from the above expressions (\ref{eq:moni}) and (\ref{eq:moni-izqe}) implies that
\begin{equation} \label{eq:index} i(\sigma, A, V) + i(\sigma, B, V) = i(\sigma, A + B, V) + i(\sigma, A\cap B, V)\, .\end{equation}

\item If $\sigma \in \gl({V},A)$ and ${\tilde V}\subset V$ is a subspace, such that ${\tilde V}$ is invariant by $\sigma$ and $A \subset {\tilde V}$, it is clear that
\begin{equation} \label{eqprop:index-recrt-equal} i(\sigma, A, {\tilde V}) = i(\sigma, A, V)\, . \end{equation}

\item If $A$ is a subspace of $V$, and $\sigma (A) = A$, then \begin{equation} \label{eqprop:index-equal} i(\sigma, A, V) = 0\, .\end{equation}

\end{enumerate}

Moreover, bearing in mind that $i(\sigma, V_+,V) = 0$ when $\text{dim}_k (V_+) < \infty$, one has that

\begin{lem}\label{lem:aldehuela-3} If $A$ and $B$ are two commensurable vector subspaces of $V$, $A\sim B$, and $\sigma \in \gl({V},A)$, then
\begin{equation} \label{eq:index-invariant} i(\sigma, A, V) = i(\sigma, B, V)\, . \end{equation}
In particular, if $V_+ \sim V$ or $V_+ \sim \{0\}$, then $i(\tau, V_+, V) = 0$ for all $\tau \in \gl (V,V_+)$.
\end{lem}

Furthermore, a direct consequence of the definition of the index and expression (\ref{eq:index}) is:

\begin{lem} \label{lem:index-izq-mon} If $V = V_+ + {\tilde V}$ and $\sigma \in \gl
(V,V_+)$ with $\sigma ({\tilde V}) = {\tilde
V}$, then $\sigma \in \gl ({\tilde V},V_+ \cap {\tilde V})$
and $$i(\sigma,{V_+},{V}) = i(\sigma,{V_+\cap
{\tilde V}},{{\tilde V}})\, .$$
\end{lem}

\subsection{Abstract Residue.}\label{ss:abst-resi}

    Let $k$ again be an arbitrary field, let $V$ be a $k$-vector space and let $A, B$ be two
$k$-vector subspaces of $V$.

    If we set $A<B$ when $(A+B)/B$ is finite dimensional, it is clear that $A$ and $B$ are commesurable, $A\sim B$, if
and only if $A<B$ and $A<B$. Commensurability is an equivalent
relation on the set of $k$-vector subspaces of $V$.

    Let us now consider an endomorphism $\varphi$ of $V$. We say
that $\varphi$ is ``finite potent" if $\varphi^n V$ is finite
dimensional for some $n$, and a trace $\tr_V(\varphi) \in k$ may
be defined, with the following properties:
\begin{enumerate}
\item If $V$ is finite dimensional, then $\tr_V(\varphi)$ is the
ordinary trace.

\item If $W$ is a subspace of $V$ such that
$\varphi W \subset W$, then $$\tr_V(\varphi) = \tr_W(\varphi) +
\tr_{V/W}(\varphi)\, .$$

\item If $\varphi$ is nilpotent, then
$\tr_V(\varphi) = 0$.
\end{enumerate}

\begin{rem}\label{r:chartraces} Properties (1), (2) and (3)
characterize traces, because if $W$ is a finite dimensional
subspace of $V$ such that $\varphi W \subset W$ and $\varphi^n V
\subset W$, for some $n$, then $\tr_V(\varphi) = \tr_W(\varphi)$.
And, since $\varphi$ is finite potent, we may take $W = \varphi^n
V$.
\end{rem}
    Let us now fix a subspace $A$ of $V$, and then define subspaces $E, E_0, E_1,
E_2$ of $\ed_k (V)$ by
\begin{itemize}
\item $\varphi \in E \Longleftrightarrow \varphi A < A$, \item
$\varphi \in E_1 \Longleftrightarrow \varphi V < A$, \item
$\varphi \in E_2 \Longleftrightarrow \varphi A < (0)$, \item
$\varphi \in E_0 \Longleftrightarrow \varphi V < A \text{ and }
\varphi A < (0)$.
\end{itemize}

    J. Tate (\cite{Ta}) showed that $E$ is a $k$-subalgebra of
$\ed (V)$, the $E_i$ are two-sided ideals in $E$, the $E$'s depend
only on the $\sim$-equivalence class of $A$, $E_1\cap E_2 = E_0$,
$E_1 + E_2 = E$, and $E_0$ is finite potent (i.e. there exists an n such that for any family
of n elements $\varphi_1, \dots, \varphi_n \in E_0$ the space $\varphi_1 \cdots \varphi_n V$ is finite dimensional). Moreover, if we
assume either $f\in E_0$ and $g\in E$, or $f\in E_1$ and $g\in
E_2$, then the commutator $[f,g] = fg - gf$ is in $E_0$ and has
zero trace.

    Let $K$ be a commutative $k$-algebra, $V$ a $K$-module, and
$A$ a $k$-subspace of $V$ such that $fA< A$ for all $f\in K$.
With the above notations, $K$ operates on $V$ through $E\subset
\ed_k (V)$.

\begin{defn}[Abstract Residue] \label{defn:abst-res} There exists a unique $k$-linear
``residue map" $$\res_A^V\colon \Omega_{K/k}^1\longrightarrow
k$$\noindent such that for each pair of elements $f$ and $g$ in
$K$ one has that $$\res_A^V (fdg) = \tr_V ([f_1,g_1])$$\noindent
for every pair of endomorphisms $f_1$ and $g_1$ in $E$ satisfying
the following conditions:
\begin{itemize}
\item Both $f\equiv f_1$ (mod $E_2$) and $g\equiv g_1$ (mod
$E_2$); \item Either $f_1 \in E_1$ or $g_1 \in E_1$.
\end{itemize}
\end{defn}

\begin{rem} \label{r:rescurve} Let $C$ be a non-singular and
irreducible curve over a perfect field $k$ and let $\Sigma_C$ be
its function field. With the notation of Remark \ref{rem:curve}, we have a map $$\res_{A_x}^{K_x}\colon \Omega_{K/k}^1 \to k\, .$$
When $x$ is a rational point of $C$, one has that $A_x \simeq
k[[t]]$, $K_x \simeq k((t))$ and $\res_x(fdg)$ is the classical
residue $\res_x(fdg)$; that is, it is the coefficient of $t^{-1}$ in
$f(t)g'(t)$.

In general, if $k(x)$ is the residue field of the closed point $x$, the classical residue $\res_x(fdg)$ values in $k(x)$, and the explicit expression
of this map is $$\res_{A_x}^{K_x} (fdg) = \tr_{k(x)/k} [\res_x(fdg)]\, ,$$\noindent where $\tr_{k(x)/k}$ is the trace of the finite extension of fields $k\hookrightarrow k(x)$.
\end{rem}

\subsubsection{Properties of $\res_A^V$.}

If $g,h \in K$, it follows from the statements of \cite{Ta} that:

\begin{enumerate}
\item \label{prop:one} If $A \subset V' \subset V$ and $KV'\subset
V'$, then
\begin{equation} \label{prop:res-t-nm} \res_A^V (gdh) = \res_A^{V'} (gdh)\, .\end{equation}

\item \label{prop:two} $\res_A^V (gdh) =
\res_{A'}^V (gdh)$ if $A\sim A'$.

\item If $g = 0$ or $h = 0$, then $\res_A^V(gdh) = 0$.

\item Let $g\in K$. Then, $\res_A^V (g^n dg) = 0$ for all integers
$n\geq 0$ and, moreover, the same holds for all $n\leq -2$ if $g$
is invertible in $K$. In particular, $\res_A^V (dg) = 0$ for all
$g\in K$.

\item If $g$ is invertible in $K$, and $h\in K$ such that
$hA\subset A$, then
$$\res_A^V(hg^{-1}dg) = \tr_{A/A\cap gA}(h) - \tr_{gA/A\cap
gA}(h)\, .$$\noindent In particular, if $g$ is invertible and $gA
\subset A$, then $$\res_A^V(g^{-1}dg) = \text{ dim}_k(A/gA)\, .$$

Moreover, if $gA = A$, then \begin{equation} \label{eqprop:res-equal} \res_A^V (gdh) = 0\, .\end{equation}

\item If $B$ is another $k$-subspace of $V$ such that $fB < B$ for
all $f\in K$, then $f(A+B) < A+B$ and $f(A\cap B) < A\cap B$ for
all $f\in K$, and one has that
\begin{equation} \label{eq:basic}{\bf \res_A^V (gdh) + \res_B^V (gdh) = \res_{A+B}^V (gdh) + \res_{A\cap B}^V (gdh)}\,
.\end{equation}

 \item Let $K'$ be a commutative $K$-algebra that
is a free $K$-module of finite rank. Let $V' = K'\otimes_k V$ and
let $A' = \sum_i x_i \otimes A \subset V'$, where $(x_i)$ is a
$K$-base for $K'$. Then $f'A'\sim A'$ for all $f'\in K'$, the
$\sim$-equivalence class of $A'$ depends only on that of $A$, not
on the choice of basis $(x_i)$, and one has $$\res_{A'}^V(f'dg) =
\res_A^V((\tr_{K'/K}f')dg) \text{ for } f'\in K' \text{ and } g\in
K\, .$$
\end{enumerate}

Moreover, it follows from the above properties and from expression (\ref{eq:basic}) that

\begin{lem} \label{lem:res-izq-mon} If $V = A + {\tilde V}$ and $gA < A$, $hA < A$, with $g({\tilde V}) = h({\tilde V}) = {\tilde
V}$ or $g({\tilde V}) = h({\tilde V}) = \{0\}$ , then $g(A \cap {\tilde V}) < A \cap {\tilde V}$, $h(A \cap {\tilde V}) < A \cap {\tilde V}$
and $$\res_{A}^{V} (gdh) = \res_{A\cap
{\tilde V}}^{{\tilde V}} (gdh)\, .$$
\end{lem}

\subsection{The Segal-Wilson pairing}\label{ss:Segal-Sil-Esther}

\subsubsection{Definition on Hilbert spaces.}

Let $H$ be a separable complex Hilbert space with a given
decomposition $H = H_+\oplus H_-$ as the direct sum of two
infinite dimensional orthogonal closed subspaces.

Let us write the operators $g\in \gl (H)$ in the block form
$$g = \begin{pmatrix} a & b \\ c & d \end{pmatrix}$$\noindent with
respect to the decomposition $H = H_+\oplus H_-$. The
\textit{restricted general linear group} $\gl_{res}(H)$ is the
closed subgroup of $\gl (H)$ consisting of operators $g$ whose
off-diagonal blocks $b$ and $c$ are compact operators. The blocks
$a$ and $d$ are therefore automatically Fredholm.

    If $\Gamma$ denotes the group of continuous maps $S^1 \longrightarrow {\mathbb C}^\times$, we can consider the subgroup $\Gamma_+$ of $\Gamma$
consisting of all-real analytic functions $f\colon S^1
\longrightarrow {\mathbb C}^\times$ that extend to holomorphic
functions $f\colon D_0 \longrightarrow {\mathbb C}^\times$ in the
disc $D_0 = \{z\in {\mathbb C} : \vert z\vert \leq 1\}$\linebreak satisfying
$f(0) = 1$, and the subgroup $\Gamma_-$ of $\Gamma$ consisting of
functions $f$, which extend to non-vanishing holomorphic functions
in $D_{\infty} = \{z\in {\mathbb C} \cup \infty : \vert z\vert
\geq 1\}$ satisfying $f(\infty) = 1$.

    Let $\gl_{1}(H)$ now be the subgroup of $\gl_{res}(H)$ consisting
of invertible operators $g = \begin{pmatrix} a & b \\ c & d
\end{pmatrix} \in \gl_{res}(H)$ and where the blocks $b$ and $c$ are
of trace class. If $\gl_{1}(H)^0$ is the identity component of
$\gl_{1}(H)$, then there exists a central extension
of groups: \begin{equation*}  1 \to {\mathbb
C}^\times \longrightarrow \gl_1^{\wedge} \longrightarrow
\gl_{1}(H)^0 \to 1\, .\end{equation*}

    If we consider the open set $\gl_{1}^{reg}$ of $\gl_{1}(H)^0$
where $a$ is invertible, there exists a section $s\colon
\gl_{1}^{reg} \to \gl_1^{\wedge}$ of the projection
$\gl_1^{\wedge} \longrightarrow \gl_{1}(H)^0$ that induces a
2-cocycle $(\cdot, \cdot) \colon \gl_{1}^{reg} \times
\gl_{1}^{reg} \longrightarrow {\mathbb C}^\times$, defined as:
\begin{equation}\label{eq:s-g-cocy} (g_1,g_2) \longmapsto \Det (a_1 a_2 a_3^{-1})\, ,\end{equation}\noindent where
$g_i =
\begin{pmatrix} a_i & b_i \\ c_i & d_i \end{pmatrix}$ and $g_3 =
g_1 g_2$.

   Note that $\gl_{1}^{reg}$ is not a group, and the map $s$ is
not multiplicative.

    Let us now consider the subgroup  $\gl_{1}^{+}$ of
$\gl_{1}^{reg}$ consisting of elements whose block decomposition
has the form $\begin{pmatrix} a & b \\ 0 & d
\end{pmatrix}$. Thus, the restriction of $s$ to $\gl_{1}^{+}$ is an inclusion of
groups $\gl_{1}^{+} \hookrightarrow \gl_1^{\wedge}$ and one can
regard $\gl_{1}^{+}$ as a group of automorphisms of the bundle
$\det$. Similar remarks apply to the subgroup $\gl_{1}^{-}$
consisting of the elements of $\gl_{1}^{reg}$ whose block
decomposition has the form $\begin{pmatrix} a & 0 \\ c & d
\end{pmatrix}$.

    Now, for every subgroup $G,{\tilde G} \subset \gl_{1}^{reg}$ such that the action of $G_1$ and $G_2$
commute with each other it is possible to define a map
\begin{equation*} \label{eq:SW-pairing} \begin{aligned}(\cdot,\cdot)_{G,{\tilde G}}^{SW}\colon G
\times {\tilde G} &\longrightarrow {\mathbb C}^\times \\
(g,{\tilde g}) &\longmapsto \Det (a{\tilde a}a^{-1}{\tilde
a}^{-1})\, ,\end{aligned}\end{equation*}\noindent where $g =
\begin{pmatrix} a & b \\ c & d
\end{pmatrix}\in G$ and ${\tilde g} =
\begin{pmatrix} {\tilde a} & {\tilde b} \\ {\tilde c} & {\tilde d}
\end{pmatrix}\in {\tilde G}$. We have that the commutator has a determinant because, from the fact that
$g$ and $\tilde g$ commute, it is equal to $1 - b{\tilde
c}a^{-1}{\tilde a}^{-1} + {\tilde b}ca^{-1}{\tilde a}^{-1}$, and
b, c, ${\tilde b}$ and ${\tilde c}$ are of trace class by the
definition of $\gl_{1}^{reg}$.

Hence this map is well-defined, and
we shall call it ``the Segal-Wilson pairing'' associated with
$G$ and ${\tilde G}$.

    Thus, if $g\in \Gamma_+$ and ${\tilde g}\in \Gamma_-$, with the above block decomposition,  a
computation shows that:
\begin{equation}\label{e:SWpairing}
({\tilde g}, g)_{\Gamma_-,
\Gamma_+}^{SW}= \Det ({\tilde a} a {\tilde a}^{-1} a^{-1}) = \exp
(\text{trace} [\alpha, {\tilde \alpha}])\, ,
\end{equation}

\noindent where $g =
\exp (f)$, ${\tilde g} = \exp ({\tilde f})$, and $\alpha$ and
$\tilde \alpha$ are the $H_+ \to H_+$ blocks of $f$ and $\tilde f$
respectively. For details, readers are referred to \cite{SW}.

\subsubsection{Definition on arbitrary vector spaces.} \label{sss:arbitrary-Segal-Wilson}

    Let $k$ be a field of characteristic zero. Similar to Subsection \ref{ss:abst-resi}, let $K$ be a commutative $k$-algebra, $V$ a $K$-module, and
$A$ a $k$-subspace of $V$ such that $fA< A$ for all $f\in K$.

    Let $k((z))$ be the field of Laurent series and let us denote $V_{z}:=V\otimes_{k}k((z))$ as a $k((z))$-vector space.

    Let $\varphi \in \ed_{k}(V)$ be a finite potent endomorphism and let us denote $\varphi \otimes 1$ the induced endomorphism of $V_{z}$. Then, there exists a well-defined exponential map:
\begin{align*}
\exp_{z} \colon \ed_{k}(V) & \to \aut_{k((z))}(V_{z})\\
\varphi & \mapsto \exp_{z}(\varphi)=\id \otimes 1+z(\varphi \otimes 1) +\frac{z^2(\varphi \otimes 1)^2}{2}+\frac{z^3(\varphi \otimes 1)^3}{3!}+\cdots
\end{align*}
Moreover, the endomorphism of $V_{z}$:
$$\bar \varphi=\exp_{z}(\varphi)-\id \otimes 1= z(\varphi \otimes 1) +\frac{z^2(\varphi \otimes 1)^2}{2}+\frac{z^3(\varphi \otimes 1)^3}{3!}+\cdots $$
is finite potent.

From the definition of determinants for finite potent endomorphisms offered in \cite{HP}, we have a well-defined determinant:

$$\Det_{V_{z}}^{k((z))}\big (\exp_{z}(\varphi)\big)=\exp_{z}(\tr_{V}(\varphi))\in k((z))^\times \,.$$

    Hence, for each $f,g\in K$, the function:
 \begin{align*}
c_{V_{+}}^V\colon K\times K&\to k((z))^\times \\
(f,g)&\mapsto c_{V_{+}}^V(f,g):=\Det_{V_{z}}^{k((z))}\big(\exp_{z}(f_{1})\exp_{z}(g_{1})\exp_{z}(-(f_{1}+g_{1}))\big)
\end{align*}
is a $2$-cocycle of $K$ with coefficients in $k((z))^\times$, for every pair endomorphisms $f_{1}$ and $g_{1}$  in $E$ satisfying:
\begin{itemize}
\item $f\equiv f_{1}$ (mod $E_{2}$) and $g\equiv g_{1}$ (mod $E_{2}$);
\item Either $f_{1}\in E_{1}$ or $g_1\in E_{1}$.
\end{itemize}
In particular, there exists a central extension of groups:
$$1\to k((z))^\times \to  \widetilde K_{V_{+}}^V\to K\to 1\,.$$

\begin{rem}
If we write $f=\begin{pmatrix} a & b \\ c& d \end{pmatrix}$ with respect to a fixed decomposition of $V=V_{+}\oplus V_{-}$, we can assume  that $f_{1}=\begin{pmatrix} a & 0 \\ 0& 0 \end{pmatrix}$. Therefore, the definition of the cocycle $c_{V_{+}}^V$ is inspired by Segal and Wilson's cocycle of \cite[Prop.3.6]{SW} (see also equation (\ref{eq:s-g-cocy})).
Moreover:
$$\begin{aligned} c_{V_{+}}^V(f,g)&=\Det_{V_{z}}^{k((z))}\big(\exp_{z^2}(\frac{1}{2}[f_{1},g_{1}])\big)
\\ &=\exp_{z^{2}}\big(\frac{1}{2}\tr_{V}([f_{1},g_{1}])\big)=\exp_{z^{2}}\big(\frac{1}{2}\res_{V_{+}}^{V}(fdg)\big)\,.\end{aligned}$$
Thus, the cocycle $c_{V_{+}}^V$ is a multiplicative analogue for the abstract residue defined by Tate in \cite{Ta}, and has a similar shape to Segal and Wilson's pairing.
\end{rem}

The cocycle $c_{V_{+}}^V$ satisfies the following properties:

\begin{itemize}
\item If $A\sim A'$, then $c_{A}^V(f,g)=c_{A'}^V(f,g)$. That is, the cocycle depends only on the commensurability class of $V_{+}$. In particular $\widetilde K_{A}^V=\widetilde K_{A'}^V$.
\item If $f(A)\subset A$ and $g(A)\subset A$, then $c_{A}^V(f,g)=1$. In particular, the central extension $\widetilde K_{A}^V$ is trivial.
\item $c_{A}^V(1,g)=1$ for all $g\in K$.
\item  If $g$ is invertible in $K$ and $h\in K$ is such that $h(A)\subset A$, then:
$$c_{A}^V(hg^{-1},g)=\exp_{z^{2}}\big(\frac{1}{2}\tr_{A/A\cap gA}(h)\big)\cdot \exp_{z^{2}}\big(-\frac{1}{2}\tr_{gA/A\cap gA}(h)\big)\,.$$
\item  If $B$ is another $k$-subspace of $V$ such that $f(B)<B$ for all $f\in K$, then:
\begin{equation} \label{eq:aerr} {\bf c_{A+B}^V(f,g)\cdot c_{A\cap B}^V(f,g)=c_{A}^V(f,g)\cdot c_{B}^V(f,g)\,.}\end{equation}
\end{itemize}

For details, readers are referred to \cite{HP}.

\section{Main Result: General Reciprocity Law}\label{s:explicit}

Let $V$ a vector space over an arbitrary ground field $k$, and let $X$ be a set of $k$-subspaces of V. Let us assume
that $X$ is closed for sums and intersections (i.e. if $A,B \in X$, then $A+B, A\cap B \in X$).

Let $G$ be a group (with multiplicative notation).

\begin{defn} We shall use the term ``X-symbol'' to refer to a map $f\colon X \rightarrow G$ satisfying the following conditions:
\begin{itemize}
\item If $\{0\} \in X$ (resp. $V \in X$), then $f(\{0\}) = 1$ (resp. $f(V) = 1$).
\item If $A,B \in X$ and $A\sim B$, then $f(A) = f(B)$.
\item If $A,B \in X$, then \begin{equation} \label{eq:tanri} f(A)\cdot f(B) = f(A+B) \cdot f(A\cap B)\, .\end{equation}
\end{itemize}
\end{defn}

\begin{defn}\label{def:independent} A family $\{A_i\}_{i\in I} \subset X$ is called ``independent for commensurability'' when for each $i\in I$ one has that $A_i \cap [\sum_{j\ne i} A_j] \sim \{0\}$.
\end{defn}

\begin{lem} \label{l:erae} If $\{A_i\}_{i= 1}^n \subset X$ is a finite family independent for commensurability, then $$f(A_1 + \dots + A_n) = \prod_{i=1}^n f(A_i)\, .$$
\end{lem}

\begin{proof} By a routine induction, it follows from formula (\ref{eq:tanri}) that $$f(A_1 + \dots + A_n) = (\prod_{i=1}^n f(A_i)\cdot [\prod_{i=1}^{n-1} f((A_1 + \dots + A_i)\cap A_{i+1})]^{-1}\, .$$
    Since $\{A_i\}$ is independent for commensurability, then $(A_1 + \dots + A_i)\cap A_{i+1} \sim \{0\}$ and $$f((A_1 + \dots + A_i)\cap A_{i+1}) = 1$$\noindent for all i, from where the claim is deduced.
\end{proof}

\begin{thm}[General Reciprocity Law] \label{thm:GRLs} Let $V$, $X$ and $f$ be as above, and let $\{A_i\}_{i\in I} \subset X$ be a family of subspaces of $V$ such that $f(A_i) = 1$ for almost all i. Let us assume that for each finite subset $J \subset I$ there is given a subspace $B_J \in X$ such that $A_i \subset B_J$ for all $i\notin J$ and
\begin{itemize}
\item (a) if $J' \subset J$, then $B_{J'} = B_J + \sum_{i\in J-J'}A_i$;
\item (b) for each $J$, the subspaces $\vert J\vert$ + 1 subspaces $\{A_j\}_{j\in J}$ and $B_J$ are independent for commensurability;
\item (c) for each $J$, if $f(A_i) = 1$ for all $i\notin J$, then $f(B_J) = 1$.
\end{itemize}
Thus, $$f(B_{\emptyset}) = \prod_{i\in I} f(A_i)\, .$$
\end{thm}

\begin{proof} Let $J$ be large enough so that condition (c) holds. Then, by (a), (b) and Lemma \ref{l:erae} we have $$f(B_{\emptyset}) = [\prod_{j\in J} f(A_j)] \cdot f(B_J)\, .$$

    Hence, we conclude, bearing in mind that $f(B_J) = 1 = \prod_{i\notin J} f(A_i)$.
\end{proof}

\subsection{Explicit reciprocity laws for the index.} \label{ss:tame-symbol-explicit}

Let us consider an arbitrary vector $k$-space $V$, and let $\{V_i\}_{i\in I}$ be a family of $k$-subspaces of
$V$.

\begin{defn} We shall use the term ``the restricted linear group'' associated
with $\{V_i\}_{i\in I}$, $\gl (V, \{V_i\}_{i\in I})$ to refer to the subgroup of
${\aut}_k (V)$ defined by:
$$\gl (V, \{V_i\}_{i\in I}) = \{\sigma \in {\aut}_k (V) \text{
such that } \sigma \in \gl (V,V_i) \text { for all i }$$\noindent
$$\text{ and } \sigma (V_i) = V_i \text { for almost all i}\}\,
.$$
\end{defn}

    If $V_I = \prod_{i\in I} V$, the diagonal embedding induces a natural immersion of groups $\gl (V, \{V_i\}_{i\in I}) \hookrightarrow \aut_k (V_I)$.

Moreover, it follows from the commutative diagram

$$\xymatrix{V \ar@{^(->}[r]
 & V_I \\
{V_i \,} \ar@{^(->}[u] \ar@{^(->}[r] &
{\prod_{i\in I} V_i} \ar@{_(->}[u] \, ,}$$\noindent and from expression (\ref{eqprop:index-recrt-equal}) that
$$i(\sigma, {V_i},{V_I}) = i(\sigma, {V_i}, {V})\, ,$$\noindent for all $\sigma \in \gl (V,
\{V_i\}_{i\in I})$.

     Let us now consider the set of $k$-subspaces of $V_I$, $$X = \{\{0\}, V_I, \prod_{j\in J} V_j\}_{J\subseteq I}\, ,$$\noindent and the X-symbol
$$\begin{aligned} f^i_{\sigma}\colon X &\longrightarrow {\mathbb Z} \\ A &\longmapsto i(\sigma, {A},{V_I})\, ,\end{aligned}$$\noindent where $\sigma \in \gl (V,
\{V_i\}_{i\in I})$.

    If $B_J = \prod_{j\in I-J} V_j$, it is clear that the X-symbol $f^i_{\sigma,\tau}$ satisfies the hypothesis of Theorem \ref{thm:GRLs} for the family $\{V_i\}_{i\in I}$, and hence
\begin{equation} \label{eq:index-erada-3} i(\sigma, {\prod_{i\in I} V_i},{V_I}) = \sum_{i\in I} i(\sigma,{V_i},{V_I}) = \sum_{i\in I} i(\sigma,{V_i},{V})\, .\end{equation}

    Let us fix two families of subspaces of $V$, $\{V_i\}_{i\in I}$ and $\{W_i\}_{i\in
I}$, such that $V_i \subseteq W_i$ for all $i\in I$ and $\sigma
(W_i) = W_i$ for every $\sigma \in  \gl (V, \{V_i\}_{i\in I})$.

    Let us now write $A_{I} = \prod_{i\in I} V_i$,
$W_{I} =  \prod_{i\in I} W_j$. If $\sigma \in
\gl (V, \{V_i\}_{i\in I})$, it is clear that $\sigma \in \gl
(W_I, A_I)$ by means of the diagonal embedding.

    If we replace $V_I$ with $W_I$ in formula (\ref{eq:index-erada-3}), from Theorem \ref{thm:GRLs} we also obtain the following result:

\begin{cor} \label{th:index-gen-rec}  Let $\{V_i\}_{i\in I}$ and $\{W_i\}_{i\in I}$
be two families of subspaces of $V$ such that $V_i \subseteq W_i$
for all $i\in I$ and $\sigma (W_i) = W_i$ for every $\sigma \in
\gl (V, \{V_i\}_{i\in I})$. Then:
$$i(\sigma, {A_I},{W_I}) = \sum_{i\in I} i(\sigma, {V_i},{W_i})\,
,$$\noindent where only a finite number of terms are different
from 0.

Accordingly, if $i(\sigma, {A_I},{W_I}) = 0$, then:
$$\sum_{i\in I} i(\sigma, {V_i},{W_i}) = 0\,
,$$\noindent where only a finite number of terms are different
from 1.
\end{cor}

\begin{rem} Note that the statement of Corollary \ref{th:index-gen-rec} holds when $A_{I} = \oplus_{i\in I} V_i$,
$W_{I} =  \oplus_{i\in I} W_j$, $i(\sigma, {A_I},{W_I}) = 0$, for $\sigma
\in \gl (V, \{V_i\}_{i\in I})$.
\end{rem}

\subsubsection{Index reciprocity for reciprocity-admissible families of subspaces.} \label{sss:recip-tame}

Let us consider again an arbitrary infinite-dimensional vector
$k$-space $V$.

\begin{defn} We shall use the term ``reciprocity-admissible family'' to
refer to a set of subspaces $\{V_i\}_{i\in I}$ of $V$ satisfying
the following properties:
\begin{itemize}
\item $\sum_{i\in I} V_i \sim V$ \item $V_{i}  \cap [\underset
{j\ne i} \sum V_j] \sim \{0\}$ for all ${i\in I}$.
\end{itemize}

\end{defn}

\begin{exam} \label{ex:natural} If $V = \underset {n\in {\mathbb N}} \oplus <e_n>$,
each map $\phi \colon {\mathbb N} \times {\mathbb N}
\longrightarrow {\mathbb N}$ such that $$\begin{aligned} \#
\{n\in {\mathbb N} &\text{ with } n\notin Im \phi\} < \infty \\
\# \{(\alpha,\beta) \in {\mathbb N} \times {\mathbb N}  \text{ s.
t. there exists } &(\gamma, \delta) \in {\mathbb N} \times
{\mathbb N} \text{ with } \phi (\alpha,\beta) = \phi (\gamma,
\delta) \} < \infty
\end{aligned}$$\noindent determines the following reciprocity-admissible family of
subspaces of $V$:
$$V_i = \underset {j\in {\mathbb N}} \oplus <e_{\phi (i,j)}>\, .$$

    In particular, every bijection $\varphi \colon {\mathbb N} \times {\mathbb N}
\overset \sim \longrightarrow {\mathbb N}$ induces a
reciprocity-admissible family of subspaces of $V$.

    Moreover, similar to \cite{Mi}, we denote with $\gl(k)$ the
direct limit of the sequence $$\gl(1,k)\subset \gl(2,k)\subset
\gl(3,k)\subset \dots\, ,$$\noindent where each linear group
$\gl(n,k)$ is injected into $\gl(n+1,k)$ by the correspondence
$$\tau \longmapsto \phi_n^{n+1} (\tau) = \begin{pmatrix} \tau & 0 \\ 0 & 1 \end{pmatrix}\, .$$

    It is clear that
$\gl(k)$ is a subgroup of the restricted linear group $\gl (V,
\{V_i\}_{i\in I})$ associated with every reciprocity-admisible
family $\{V_i\}_{i\in {\mathbb N}}$ defined above.
\end{exam}

Let us now consider the family of $k$-subspaces of $V$
$${\tilde X} = \{\{0\}, V, V_i, V_{i_1} + \dots + V_{i_r}, V_{k_1}\cap \dots \cap V_{k_s}, \sum_{j\in J} V_j\}_{\{i,i_n,k_m\} \in I; J\subseteq I}\, .$$\noindent If $\sigma \in \gl (V,
\{V_i\}_{i\in I})$, the ${\tilde X}$-symbol
$$\begin{aligned} {\tilde f}_{\sigma}^i\colon {\tilde X} &\longrightarrow {\mathbb Z} \\ A &\longmapsto i(\sigma, {A},{V})\, ,\end{aligned}$$
satisfies the hypothesis of Theorem \ref{thm:GRLs} with $B_J = \sum_{j\in I-J} V_j$ for the family $\{V_i\}_{i\in I}$.

Hence, since $\sum_{i\in I} V_i \sim V$, then $i(\sigma, {\sum_{i\in I} V_i},{V}) = 0$ -see Lemma \ref{lem:aldehuela-3}-, and we obtain the following result:

\begin{cor} \label{th:index-rec} If $\{V_i\}_{i\in I}$ is a reciprocity-admissible family of
subspaces of $V$, for all $\sigma \in \gl
(V, \{V_i\}_{i\in I})$ one has that:
$$\sum_{i\in I} i(\sigma,V_i,V) = 0\,
,$$\noindent where only a finite number of terms are different
from 0.
\end{cor}

\subsubsection{Sum of valuations on a complete curve}\label{sss:sum-valuations}

Similar to Remark \ref{rem:curve}, let $C$ be a
non-singular and irreducible curve over a perfect field $k$, and
let $\Sigma_C$ be its function field. If $x\in C$ is a closed
point, we set $A_x = {\widehat {\O_x}}$ (the completion of the
local ring $\O_x$), and $K_x = (\widehat {\O_x})_0$ (the field of
fractions of ${\widehat {\O_x}}$ that coincides with the
completion of ${\Sigma}_C$ with respect to the valuation ring
$\O_x$).

    If we set $$V = \underset
{x\in C} {{\prod}'} K_x = \{ f = (f_x) \text { such that } f_x \in
K_x \text { and } f_x\in A_x \text { for almost all } x\}\,
,$$\noindent then $\Sigma_C$ is regarded as a subspace of $V$ by
means of the diagonal embedding.

 It is clear that $\Sigma_C^{\times}$ is a commutative subgroup of $\gl(V,\{A_x\}_{x\in C})$.

    Let us now consider the family of $k$-subspaces of $V$
consisting of $V_1 = \prod_{x\in C} A_x$ and $V_2 =
\Sigma_C$. If $C$ is a complete curve, since $$V/(V_1 + V_2)
\simeq H^1(C, {\mathcal O}_C) \text { and } V_1\cap V_2 = H^0(C,
{\mathcal O}_C)\, ,$$\noindent then $\{V_1,V_2\}$ is a
reciprocity-admissible family of $V$.

    Thus, bearing in mind that $\Sigma_C^\times \subset \gl
(V,\{V_1,V_2\})$ and that $V_2$ is invariant under the action of
$\Sigma_C^\times$, it follows from expression (\ref{eqprop:index-equal}) and Corollary \ref{th:index-rec} that
$$i(f,{V_1},V) = 0\, ,$$\noindent for all $f \in
\Sigma_C^\times$.

    Moreover, since $V = V_1 + \underset {x\in C} \oplus
K_x$, $\underset {x\in C} \oplus K_x$ is invariant under the
action of $\Sigma_C^\times$, and $V_1 \cap \underset {x\in C}
\oplus K_x = \underset {x\in C} \oplus A_x$, then it follows
from Lemma \ref{lem:index-izq-mon} that
$$i(f,{V_1},V) = i(f,{\underset {x\in C}
\oplus A_x},{\underset {x\in C} \oplus K_x}) = 0\, ,$$\noindent for all $f \in
\Sigma_C^\times$.

    Hence, setting:
\begin{itemize}
\item $V_x = \dots \oplus \{0\} \oplus A_x \oplus \{0\} \oplus
\dots$, \item $W_x = \dots \oplus \{0\} \oplus K_x \oplus \{0\}
\oplus \dots$,
\end{itemize}
the families $\{V_x\}_{x\in C}$ and $\{W_x\}_{x\in C}$ of
$k$-subspaces of $V$ satisfy the hypothesis of Corollary
\ref{th:index-gen-rec}, and since
$$i(f,A_x,K_x) = \deg (x)\cdot v_x(f)$$\noindent
for all $f \in \Sigma_C^\times$, where $k(x)$ is the residue
class field of the closed point $x$ and $\deg (x) = \text{dim}_k
(k(x))$. Thus, we have that
$$\sum_{x\in C} \deg (x)\cdot v_x(f) = 0\, ,$$\noindent which is the well-known formula for the sum of valuations on a complete curve.

\subsubsection{A reciprocity law for the symbol $\nu_{x,C}$ on a surface.} \label{ss:sum-nu-surface}

Let $C$ be an irreducible and non-singular algebraic curve on a smooth, proper,
geometrically irreducible surface $S$ over a perfect
field $k$. If ${\Sigma}_S$ is the function field of $S$, the curve
$C$ defines a discrete valuation
$$v_C\colon {\Sigma}_S^{\times} \to {\mathbb Z}\, ,$$\noindent whose
residue class field is ${\Sigma}_C$ (the function field of $C$).

    Thus, it follows from the characterization of the tame symbol offered in \cite{Pa} that there
exists a central extension of groups \begin{equation} \label{eq:surfac-curve} 1 \to {\Sigma}_C^{\times} \to
{\tilde {{\Sigma}_S^{\times}}} \to {\Sigma}_S^{\times}\to 1\, ,\end{equation}\noindent such
that the value of its commutator is: $$\{f, g\}_C^S =
(\frac{f^{v_C(g)}}{g^{v_C(f)}})_{\vert_C}\in {\Sigma}_C^{\times}\,
,$$\noindent for each $f,g \in {\Sigma}_S^*$.

Moreover, if we fix a function $z\in
{\Sigma}_S^{\times}$ such that $v_C(z) = 1$, we have a morphism of groups
$$\begin{aligned}\phi_z \colon {\Sigma}_S^{\times} &\longrightarrow {\Sigma}_C^{\times} \\ f &\longmapsto \{f, z\}_C^S \, .\end{aligned}$$

Note that $\phi_z (f) = [f\cdot z^{-v_c(f)}]_{\vert_C}$.

If $x\in C$ and ${\overline {v_x^z}} (f) = v_x[\phi_z(f)]$, we have that ${\overline {v_x^z}}$ depends on the choice of the parameter $z$, because
  if $v_C(z') = 1$ then $${\overline {v_x^{z'}}} (f) = {\overline {v_x^{z}}} (f) + \lambda_{z',z}\cdot v_C(f)\, ,$$\noindent with $\lambda_{z',z} = v_x[(\frac{z'}z)_{\vert_C}] \in {\mathbb Z}$.

  Moreover, it is clear that \begin{equation} \label{eq:osipov-1213} i(\phi_z(f),A_x,K_x) = \deg(x)\cdot {\overline {v_x^z}} (f)\, ,\end{equation}\noindent where $k(x)$ is the residue class field of $x$ and $\deg(x) = \text{dim}_k k(x)$.

    Let consider the map $\nu_{x,C}\colon \Sigma_S^\times\times \Sigma_S^\times \longrightarrow {\mathbb Z}$ defined by the formula $$\nu_{x,C} (f,g) =
\begin{vmatrix} {\overline {v_x^z}} (f) & {\overline {v_x^z}} (g) \\ v_C(f) & v_C(g)\end{vmatrix}\, .$$

The symbol $\nu_{x,C}$ is independent of the choice of the parameter $z$. This symbol has been used to describe the intersection index of divisors on algebraic surfaces and it also appears in the descripcion of non-ramified extensions of two-dimensional local fields when the field $k$ is finite.

Bearing in mind expression (\ref{eq:osipov-1213}), with similar arguments to above, it follows from Corollary \ref{th:index-gen-rec} and Corollary \ref{th:index-rec} that $$\sum_{x\in C} \deg(x)\cdot \nu_{x,C} (f,g) = 0 \qquad \text{ for all } f,g\in \Sigma_S^\times\, ,$$\noindent which generalizes a reciprocity law offered by D. V. Osipov in \cite{Os}.

\subsection{Explicit reciprocity laws for the tame symbol on vector spaces.} \label{ss:tame-symbol-explicit}
Let us again consider an arbitrary infinite-dimensional vector $k$-space $V$, and let $\{V_i\}_{i\in I}$ be a family of $k$-subspaces of
$V$.

Keeping the previous notation, the restricted linear group associated
with $\{V_i\}_{i\in I}$, $\gl (V, \{V_i\}_{i\in I})$, is the subgroup of
${\aut}_k (V)$ defined by:
$$\gl (V, \{V_i\}_{i\in I}) = \{\sigma \in {\aut}_k (V) \text{
such that } \sigma \in \gl (V,V_i) \text { for all i }$$\noindent
$$\text{ and } \sigma (V_i) = V_i \text { for almost all i}\}\,
.$$

    If $V_I = \prod_{i\in I} V$, the diagonal embedding induces a natural immersion of groups $\gl (V, \{V_i\}_{i\in I}) \hookrightarrow \aut_k (V_I)$.

Moreover, similar to above, it follows from expression (\ref{eqprop:recrt-equal}) that
$$<\sigma, \tau>_{V_i}^{V_I} = <\sigma, \tau>_{V_i}^{V}\, ,$$\noindent for all commuting $\sigma, \tau \in \gl (V,
\{V_i\}_{i\in I})$.

     Let us again consider the set of $k$-subspaces of $V_I$, $$X = \{\{0\}, V_I, \prod_{j\in J} V_j\}_{J\subseteq I}\, ,$$\noindent and the X-symbol
$$\begin{aligned} f_{\sigma,\tau}\colon X &\longrightarrow k^{\times} \\ A &\longmapsto <\sigma, \tau>_{A}^{V_I}\, ,\end{aligned}$$\noindent where $\sigma, \tau$ are commuting elements of $\gl (V,
\{V_i\}_{i\in I})$.

    If $B_J = \prod_{j\in I-J} V_j$, it is clear that the X-symbol $f_{\sigma,\tau}$ satisfies the hypothesis of Theorem \ref{thm:GRLs} for the family $\{V_i\}_{i\in I}$, and hence
\begin{equation} \label{eq:erada-3} <\sigma, \tau>_{\prod_{i\in I} V_i}^{V_I} = \prod_{i\in I} <\sigma, \tau>_{V_i}^{V_I} = \prod_{i\in I} <\sigma, \tau>_{V_i}^{V}\, .\end{equation}

    Let us again fix two families of subspaces of $V$, $\{V_i\}_{i\in I}$ and $\{W_i\}_{i\in
I}$, such that $V_i \subseteq W_i$ for all $i\in I$ and $\sigma
(W_i) = W_i$ for every $\sigma \in  \gl (V, \{V_i\}_{i\in I})$.

    Let us again write $A_{I} = \prod_{i\in I} V_i$,
$W_{I} =  \prod_{i\in I} W_j$. If $\sigma, \tau \in
\gl (V, \{V_i\}_{i\in I})$, it is clear that $\sigma, \tau \in \gl
(W_I, A_I)$ by means of the diagonal embedding.

    Again replacing $V_I$ with $W_I$ in formula (\ref{eq:erada-3}), from Theorem \ref{thm:GRLs} we also obtain the following result:

\begin{cor} \label{th:gen-rec}  Let $\{V_i\}_{i\in I}$ and $\{W_i\}_{i\in I}$
be two families of subspaces of $V$ such that $V_i \subseteq W_i$
for all $i\in I$ and $\sigma (W_i) = W_i$ for every $\sigma \in
\gl (V, \{V_i\}_{i\in I})$. For commuting elements $\sigma, \tau
\in \gl (V, \{V_i\}_{i\in I})$, we have:
$$<\sigma,
\tau>^{W_I}_{A_I} = \prod_{i\in I} <\sigma, \tau>^{W_i}_{V_i}\,
,$$\noindent where only a finite number of terms are different
from 1.

Accordingly, if $\{\sigma,
\tau\}^{W_I}_{A_I} = 1$ and $i(\sigma, A_I, W_I) = 0$, then:
$$\prod_{i\in I} <\sigma, \tau>^{W_i}_{V_i} = 1\,
,$$\noindent where only a finite number of terms are different
from 1.
\end{cor}

\begin{rem} Note that the statement of Corollary \ref{th:gen-rec} holds when $A_{I} = \oplus_{i\in I} V_i$,
$W_{I} =  \oplus_{i\in I} W_j$, $\{\sigma,
\tau\}^{W_I}_{A_I} = 1$ and $i(\sigma, A_I, W_I) = 0$ for commuting elements $\sigma, \tau
\in \gl (V, \{V_i\}_{i\in I})$.
\end{rem}

    Let us now consider a commutative group $G$ and a morphism of
groups \linebreak $\varphi\colon k^\times \to G$. Hence, the
central extension (\ref{e:greg}) induces  an exact sequence of
groups
$$1\to G \to {\widetilde {(\gl {V,V_+}})_{\varphi}} \to \gl (V,V_+) \to 1$$
and, given  commuting elements $\sigma, \tau \in \gl (V,V_+)$, we
shall denote  the corresponding commutator by $\{\sigma,\tau\}_{V_+,\varphi}^{V}$.

    For each commuting $\sigma, \tau \in \gl (V,
\{V_i\}_{i\in I})$, if $X$ is the above family of $k$-subspaces of $V_I$, from the X-symbol
$$\begin{aligned} f_{\sigma,\tau}^{\varphi}\colon X &\longrightarrow G \\ A &\longmapsto <\sigma, \tau>_{A,\varphi}^{V_I}\, ,\end{aligned}$$\noindent with the same arguments of Corollary \ref{th:gen-rec} one has that:

\begin{cor} \label{cor:gen-rec-vector} Let $\{V_i\}_{i\in I}$ and $\{W_i\}_{i\in I}$
be two families of subspaces of $V$ such that $V_i \subseteq W_i$
for all $i\in I$ and $\sigma (W_i) = W_i$ for every $\sigma \in
\gl (V, \{V_i\}_{i\in I})$. For commuting elements $\sigma, \tau
\in \gl (V, \{V_i\}_{i\in I})$, if $G$ is a commutative group and $\varphi\colon k^*
\to G$ is a morphism of groups, then:
$$ \varphi (<\sigma, \tau>^{W_I}_{A_I}) = \prod_{i\in I} [\varphi (-1)]^{i(\sigma,V_i,W_i)i(\tau,V_i,W_i)}\cdot \{\sigma, \tau\}^{W_i}_{V_i,\varphi}\,
,$$\noindent where only a finite number of terms are different
from 1.

Accordingly, if $\{\sigma,
\tau\}^{W_I}_{A_I} = 1$ and $i(\sigma, A_I, W_I) = 0$, then:
$$\prod_{i\in I} [\varphi (-1)]^{i(\sigma,V_i,W_i)i(\tau,V_i,W_i)}\cdot \{\sigma, \tau\}^{W_i}_{V_i,\varphi} = 1\,
,$$\noindent where only a finite number of terms are different
from 1.
\end{cor}

\begin{rem} Similar to the above, it is clear that the statement of Corollary \ref{cor:gen-rec-vector} holds when $A_{I} = \oplus_{i\in I} V_i$,
$W_{I} =  \oplus_{i\in I} W_j$, $\{\sigma,
\tau\}^{W_I}_{A_I} = 1$ and $i(\sigma, A_I, W_I) = 0$ for commuting elements $\sigma, \tau
\in \gl (V, \{V_i\}_{i\in I})$.
\end{rem}

\subsubsection{Tame Reciprocity for reciprocity-admissible families of subspaces.} \label{sss:recip-tame}

Let us consider again an arbitrary vector
$k$-space $V$ and a reciprocity-admissible family $\{V_i\}_{i\in I}$ of subspaces of
$V$. Recall that this set satisfies the following properties:
\begin{itemize}
\item $\sum_{i\in I} V_i \sim V$ \item $V_{i}  \cap [\underset
{j\ne i} \sum V_j] \sim \{0\}$ for all ${i\in I}$.
\end{itemize}

 If $\{V_i\}_{i\in I}$ is a reciprocity-admissible family of
subspaces of $V$, and we now consider the family of $k$-subspaces of $V$
$${\tilde X} = \{\{0\}, V, V_i, V_{i_1} + \dots + V_{i_r}, V_{k_1}\cap \dots \cap V_{k_s}, \sum_{j\in J} V_j\}_{\{i,i_n,k_m\} \in I; J\subseteq I}\, ,$$\noindent where $\sigma, \tau \in \gl (V,
\{V_i\}_{i\in I})$ are commuting elements, the ${\tilde X}$-symbol
$$\begin{aligned} {\tilde f}_{\sigma,\tau}\colon {\tilde X} &\longrightarrow k^{\times} \\ V_i &\longmapsto <\sigma, \tau>_{V_i}^{V}\, ,\end{aligned}$$
satisfies the hypothesis of Theorem \ref{thm:GRLs} with $B_J = \sum_{j\in I-J} V_j$ for the family $\{V_i\}_{i\in I}$.

Hence, since $\sum_{i\in I} V_i \sim V$, then $<\sigma, \tau>_{\sum_{i\in I} V_i}^{V} = 1$ -see expression (\ref{eprop:rqerq})-, and we get the following result:

\begin{cor} \label{th:rec} If $\{V_i\}_{i\in I}$ is a reciprocity-admissible family of
subspaces of $V$, for all commuting elements $\sigma, \tau \in \gl
(V, \{V_i\}_{i\in I})$ one has that:
$$\prod_{i\in I} (-1)^{i(\sigma,V_i,V)i(\tau,V_i,V)}\cdot \{\sigma, \tau\}^{V}_{V_i} = 1\,
,$$\noindent where only a finite number of terms are different
from 1.
\end{cor}

    Furthermore, similar to Corollary \ref{cor:gen-rec-vector}, we also have that:

\begin{cor} \label{cor:rec-vector} If $\{V_i\}_{i\in I}$ is a reciprocity-admissible family of
subspaces of $V$, $\sigma, \tau \in \gl (V, \{V_i\}_{i\in I})$
such that $\sigma \circ \tau = \tau \circ \sigma$, $G$ is a
commutative group, and $\varphi\colon k^\times \to G$ is a
morphism of groups, then:
$$\prod_{i\in I} [\varphi (-1)]^{i(\sigma,V_i,V)i(\tau,V_i,V)}\cdot \{\sigma, \tau\}^{V}_{V_i,\varphi} = 1\,
,$$\noindent where only a finite number of terms are different
from 1.
\end{cor}

\medskip

\subsubsection{Weil Reciprocity Law} \label{sss:curve} Let $C$ again be a
non-singular and irreducible curve over a perfect field $k$, and
let $\Sigma_C$ be its function field. If $x\in C$ is a closed
point, we set $A_x = {\widehat {\O_x}}$ (the completion of the
local ring $\O_x$), and $K_x = (\widehat {\O_x})_0$ (the field of
fractions of ${\widehat {\O_x}}$, which coincides with the
completion of ${\Sigma}_C$ with respect to the valuation ring
$\O_x$).

    If we set $$V = \underset
{x\in C} {{\prod}'} K_x = \{ f = (f_x) \text { such that } f_x \in
K_x \text { and } f_x\in A_x \text { for almost all } x\}\,
,$$\noindent then $\Sigma_C$ is regarded as a subspace of $V$ by
means of the diagonal embedding.

    It is clear that $\Sigma_C^{\times}$ is a commutative subgroup of $\gl(V,\{A_x\}_{x\in C})$.

    Let us consider the family of $k$-subspaces of $V$
consisting of $V_1 = \prod_{x\in C} A_x$ and $V_2 =
\Sigma_C$. If $C$ is a complete curve, since $$V/(V_1 + V_2)
\simeq H^1(C, {\mathcal O}_C) \text { and } V_1\cap V_2 = H^0(C,
{\mathcal O}_C)\, ,$$\noindent then $\{V_1,V_2\}$ is a
reciprocity-admissible family of $V$. Also, since $i(f,A_x,K_x) = \deg (x)\cdot v_x(f)$, we have that
$$i(h,V_1,V) = \sum_{x\in C} i(h,A_x,K_x) = \sum_{x\in C} \deg (x)\cdot v_x(h) = 0$$\noindent for all $h\in \Sigma_C^{\times} \subseteq \gl (V,V_1)$.

    Thus, bearing in mind that $\Sigma_C^\times \subset \gl
(V,\{V_1,V_2\})$ and that $V_2$ is invariant under the action of
$\Sigma_C^\times$, expression (\ref{eqprop:equal}) and Corollary \ref{th:rec} show that
$$\{f,g\}_{V_1}^V = 1\, ,$$\noindent for all $f,g \in
\Sigma_C^\times$.

    Moreover, since $V = V_1 + \underset {x\in C} \oplus
K_x$, $\underset {x\in C} \oplus K_x$ is invariant under the
action of $\Sigma_C^\times$, and $V_1 \cap \underset {x\in C}
\oplus K_x = \underset {x\in C} \oplus A_x$, then it follows
from Lemma \ref{lem:izq-mon} that
$$\{f,g\}_{V_1}^V = \{f,g\}_{\underset {x\in C}
\oplus A_x}^{\underset {x\in C} \oplus K_x} = 1\, ,$$\noindent for all $f,g \in
\Sigma_C^\times$.

    Furthermore, one has that $$i(h, \underset {x\in C}
\oplus A_x, \underset {x\in C}
\oplus K_x) =  \sum_{x\in C} \deg (x)\cdot v_x(h) = 0 \text{ for all } h\in \Sigma_C^{\times}\, .$$

    Hence, setting:
\begin{itemize}
\item $V_x = \dots \oplus \{0\} \oplus A_x \oplus \{0\} \oplus
\dots$, \item $W_x = \dots \oplus \{0\} \oplus K_x \oplus \{0\}
\oplus \dots$,
\end{itemize}
the families $\{V_x\}_{x\in C}$ and $\{W_x\}_{x\in C}$ of
$k$-subspaces of $V$ satisfy the hypothesis of Corollary
\ref{th:gen-rec}, and since
$$\{{f}, {g}\}_{A_x}^{K_x} =
N_{k(x)/k}[\frac{{f}^{v_x({g})}}{{g}^{v_x({f})}}(x)]$$\noindent
for all $f,g \in \Sigma_C^\times$, where $k(x)$ is the residue
class field of the closed point $x$, $\deg (x) = \text{dim}_k
(k(x))$ and $N_{k(x)/k}$ is the norm of the extension
$k\hookrightarrow k(x)$, then we have that
$$\prod_{x\in C} (-1)^{\deg (x)v_x(f) v_x(g)}\cdot N_{k(x)/k}[\frac{{f}^{v_x({g})}}{{g}^{v_x({f})}}(x)] = 1 \in
k^\times\, ,$$\noindent which is the explicit expression of the
Weil Reciprocity Law.

    Furthermore, with the same method one can likewise see that the reciprocity laws for
the generalizations of the tame symbol to $\gl(n,\Sigma_C)$ and
$\gl(\Sigma_C)$, offered by the author in \cite{Pa1}, where $C$ is
again a complete, irreducible and non-singular curve over a
perfect field, can also be deduced from Corollary \ref{th:gen-rec}
and Corollary \ref{th:rec}.

\begin{rem} Similar to \cite{ACK} and \cite{Pa2}, in this work we deduce the Weil Reciprocity Law from Theorem \ref{thm:GRLs} by using local arguments. We should note that a ``global'' proof of the expression
$$\{f,g\}_{\underset {x\in C}
\oplus A_x}^{\underset {x\in C} \oplus K_x} = 1\qquad \text{ for all } f,g \in
\Sigma_C^\times\, ,$$\noindent was offered by J. M. Mu{\~n}oz Porras and the author in \cite{MPa}.
\end{rem}

\medskip

\subsubsection{Hilbert Norm Residue Symbol}\label{sss:Hilbert-norm-reciprocity} Let $C$ now be an irreducible, complete and non-singular curve
over a finite perfect field $k$ that contains the $m^{th}$ roots
of unity.

    If $\# k = q$, one has the morphism of groups
$$\aligned {\phi}_m\colon k^\times &\longrightarrow {\mu}_m \\ x
&\longmapsto x^{\frac {q-1}{m}},\endaligned$$ \noindent and hence
$$\{f,g\}_{A_x,{\phi}_m}^{K_x} =  N_{k(x)/k}
\big [ \frac {f^{v_x(g)}}{g^{v_x(f)}} (x) \big ]^{\frac {q-1}{m}}
\in {\mu}_m.$$

    Arguing as above, from Corollary \ref{cor:gen-rec-vector} and
Corollary \ref{cor:rec-vector} we can deduce that
$$\begin{aligned}  \prod_{x\in C} [{\phi}_m (-1)]^{i(f,A_x,K_x)\cdot i(g,A_x,K_x)}\cdot
\{f,g\}_{A_x,{\phi}_m}^{K_x} &= \\ \prod_{x\in C} (-1)^{\deg
(x)v_x(f) v_x(g){\frac {q-1}{m}}}  N_{k(x)/k} \big [ \frac
{f^{v_x(g)}}{g^{v_x(f)}} (x) \big ]^{\frac {q-1}{m}}  &= \\
\prod_{x\in C} N_{k(x)/k} \big [ (-1)^{v_x(f)\cdot v_x(g)}\cdot
\frac {f^{v_x(g)}}{g^{v_x(f)}} (x) \big ]^{\frac {q-1}{m}} &= 1\,
,\end{aligned}$$\noindent which is the explicit formula of the
Reciprocity Law for the Hilbert Norm Residue Symbol given by H. L.
Schmidt in \cite{Sc}. And the same occurs with the reciprocity
laws for the generalizations of the Hilbert Norm Residue Symbol to
$\gl(n,\Sigma_C)$ and $\gl(\Sigma_C)$, also offered by the author
in \cite{Pa1}.

\bigskip

\subsubsection{Horozov Symbol and Parshin Symbol on a Surface.} \label{ss:Horozov-Parshin}

 Let $C$ be an irreducible and non-singular algebraic curve on a smooth, proper,
geometrically irreducible surface $S$ over an algebraically closed
field $k$. Similar to Subsection \ref{ss:sum-nu-surface}, let ${\Sigma}_S$ be the function field of $S$, and let
$v_C\colon {\Sigma}_S^{\times} \to {\mathbb Z}$ be the valuation induced by $C$, whose
residue class field is ${\Sigma}_C$ (the function field of $C$).

    If we fix a function $z\in
{\Sigma}_S^{\times}$ such that $v_C(z) = 1$, we can again consider a morphism of groups
$$\begin{aligned}\phi_z \colon {\Sigma}_S^{\times} &\longrightarrow {\Sigma}_C^{\times} \\ f &\longmapsto \{f, z\}_C^S \, ,\end{aligned}$$\noindent
i.e. $\phi_z (f) = [f\cdot z^{-v_c(f)}]_{\vert_C}$.

If $x\in C$ is a closed point and $<\cdot, \cdot>_x\colon \Sigma_C^{\times} \times \Sigma_C^{\times} \to k^\times$ is the tame symbol
associated with $x$, recently I. Horozov has defined in \cite{Ho} the following symbol:
$$\begin{aligned} (\cdot, \cdot, \cdot)_{C,x}^z\colon \Sigma_S^{\times}\times \Sigma_S^{\times} \times \Sigma_S^{\times} &\longrightarrow k^{\times} \\
(f,g,h) &\longmapsto <\phi_z(f), \phi_z(h)>_x^{v_C(g)} \cdot <\phi_z(g), -1>_x^{v_C(f)\cdot v_C(h)}\, .\end{aligned}$$

    This symbol depends on the choice of the parameter $z$ and it satisfies the reciprocity law:
    $$\prod_{x\in C} (f, g, h)_{C,x}^z = 1\, ,$$\noindent for all $f,g,h \in \Sigma_S^{\times}$.

    Arguing similar to the above, this reciprocity law can also be deduced from Corollary \ref{cor:gen-rec-vector} and
Corollary \ref{cor:rec-vector}.

\bigskip

    Moreover, if ${\overline {v_x^z}} (f) = i(\phi_z(f),A_x,K_x)$ for $x\in C$ and $f\in \Sigma_S^{\times}$,
given three functions $f,g,h \in \Sigma_S^{\times}$, the expression  $<f,g,h>_{(x,C)}$ = $$(-1)^{{\alpha}_{(x,C)}}\big
(\frac{f^{v_C(g)\cdot {\overline {v_x^z}}(h) - v_C(h)\cdot
{\overline {v_x^z}}(g)} }{g^{v_C(f)\cdot {\overline {v_x^z}}(h)-
v_C(h)\cdot {\overline {v_x^z}}(f)}}\cdot h^{v_C(f)\cdot {\overline
{v_x^z}}(g) - v_C(g)\cdot {\overline {v_x^z}}(f)}\big )_{\vert_C}
(p)\, ,$$\noindent where
$$\begin{aligned} {\alpha}_{(x,C)} &= v_C(f)\cdot v_C(g)\cdot {\overline {v_x^z}}(h) +
v_C(f)\cdot v_C(h)\cdot {\overline {v_x^z}}(g) +  v_C(g)\cdot
v_C(h)\cdot {\overline {v_x^z}}(f) +
\\ &+ v_C(f)\cdot {\overline {v_x^z}}(g)\cdot {\overline {v_x^z}}(h) +
v_C(g)\cdot {\overline {v_x^z}}(f)\cdot {\overline {v_x^z}}(h) +
v_C(h)\cdot {\overline {v_x^z}}(f)\cdot {\overline {v_x^z}}(g) \,
,\end{aligned}$$\noindent is the symbol introduced in 1985 by A. N. Parshin -\cite{Par}-. This symbol is independent of
the choice of the parameter $z$.

    In fact, the Horozov symbol is a refinement of the Parshin symbol, because
$$<f,g,h>_{(x,C)} = (f, g, h)_{C,x}^z \cdot (h, f, g)_{C,x}^z \cdot (g, h, f)_{C,x}^z\, ,$$ for all $f,g,h \in \Sigma_S^{\times}$.

    Thus, the reciprocity law for the Parshin symbol on a surface $$\prod_{x\in C} <f,g,h>_{(x,C)} = 1\, ,$$\noindent is also a direct
consequence of the statements of Corollary \ref{cor:gen-rec-vector} and
Corollary \ref{cor:rec-vector}.

\subsubsection{Horozov-Kerr 4-function local symbol.}\label{ss:Horozov-Kerr}

Similarly to the above, let $C$ again be an irreducible and non-singular algebraic curve on a smooth, proper,
geometrically irreducible surface $S$ over an algebraically closed
field $k$.

    If ${\Sigma}_S$ is the function field of $S$, $v_C\colon {\Sigma}_S^{\times} \to {\mathbb Z}$ is the discrete valuation induced by $C$; $\{\cdot,\cdot\}_C^S$ is the commutator of the central extension of groups (\ref{eq:surfac-curve}), and
${\Sigma}_C$  is the function field of $C$, we have the tame symbol: $$\begin{aligned} <\cdot, \cdot>_C \colon \Sigma_S^{\times} \times \Sigma_S^{\times} &\longrightarrow \Sigma_C^\times \\ (f,g) &\longmapsto (-1)^{v_C(f)\cdot v_C(g)} \cdot (\frac{f^{v_C(g)}}{g^{v_C(f)}})_{\vert_C}\, ,\end{aligned}$$\noindent where $\{f,g\}_C^S = (\frac{f^{v_C(g)}}{g^{v_C(f)}})_{\vert_C}$.

    Let us again consider a parameter $z\in {\Sigma}_S^{\times}$ such that $v_C(z) = 1$. As above, for every $f\in \Sigma_S^\times$,
we have $\phi_z (f) = \{f,z\}_C^S \in \Sigma_C^\times$.

    If $\langle \cdot, \cdot \rangle_x\colon \Sigma_C^{\times} \times \Sigma_C^{\times} \to k^\times$ is the tame symbol
associated with a point $x\in C$, the Horozov-Kerr 4-function local symbol $\{\cdot,\cdot,\cdot,\cdot\}_{C,x} \colon \Sigma_S^{\times} \times \Sigma_S^{\times} \times \Sigma_S^{\times} \times \Sigma_S^{\times} \rightarrow k^{\times}$ (see Definition 3.6 and appendix in \cite{HoK}) can be defined as follows:
$$ \{f_1,f_2,f_3,f_4\}_{C,x} = [\prod_{i=1}^4 \langle \phi_z(f_i),-1\rangle_x^{\prod_{j\ne i}v_C(f_j)} ]\cdot \langle <f_1,f_2>_C, <f_3,f_4>_C \rangle_x\, ,$$\noindent for all $f_1,f_2,f_3,f_4 \in \Sigma_S^\times$. It is clear that this symbol is independent of the choice of the parameter $z$.

Thus, the reciprocity law $$\prod_{x\in C} \{f_1,f_2,f_3,f_4\}_{C,x} = 1$$ can be deduced from the statements of Corollary \ref{th:gen-rec}, Corollary \ref{cor:gen-rec-vector}, Corollary \ref{th:rec} and Corollary \ref{cor:rec-vector}.

\subsection{Explicit reciprocity laws for the abstract residue.} \label{ss:abstract-residue-explicit}

    Keeping the notation of Subsection \ref{ss:abst-resi}, let $V$ be an arbitrary $k$-vector space. For each $k$-subspace $A\subseteq V$ we
write $$E(V,A) = \{g \in {\ed}_k (V) \text{
such that } g (A) < A\}\, .$$

    Let $\{V_i\}_{i\in I}$ again be a family of $k$-subspaces of
$V$.

\begin{defn} We shall use the term ``the restricted group of endomorphisms'' associated
with $\{V_i\}_{i\in I}$, $E (V, \{V_i\}_{i\in I})$, to refer to the subgroup of
${\ed}_k (V)$ defined by:
$$E (V, \{V_i\}_{i\in I}) = \{g \in {\ed}_k (V) \text{
such that } g (V_i) < V_i \text { for all i, }$$\noindent
$$\text{ and } g(V_i) = V_i \text{ or } g(V_i) = \{0\} \text { for almost all i}\}\,
.$$
\end{defn}

    If $V_I = \prod_{i\in I} V$, the diagonal embedding induces a natural immersion of groups $E (V, \{V_i\}_{i\in I}) \hookrightarrow \ed_k (V_I)$.

Let $K$ be a $k$-algebra such that $K$ is a subgroup of $E (V, \{V_i\}_{i\in I})$.

Thus, similar to the above, from the commutative diagram

$$\xymatrix{V \ar@{^(->}[r]
 & V_I \\
{V_i \,} \ar@{^(->}[u] \ar@{^(->}[r] &
{\prod_{i\in I} V_i} \ar@{_(->}[u] \, ,}$$\noindent

and from property (\ref{prop:res-t-nm}), we have that $$\res_{V_i}^{V_I} (gdh) = \res_{V_i}^{V} (gdh)\, ,$$\noindent for all $g,h \in K$.

Let us again consider the set of $k$-subspaces of $V_I$, $$X = \{\{0\}, V_I, \prod_{j\in J} V_j\}_{J\subseteq I}\, ,$$\noindent and the X-symbol
$$\begin{aligned} f_{g,h}\colon X &\longrightarrow k \\ A &\longmapsto \res_{A}^{V_I} (gdh)\, ,\end{aligned}$$\noindent where $g,h \in K$.

If $B_J = \prod_{j\in I-J} V_j$, it is clear that the X-symbol $f_{g,h}$ satisfies the hypothesis of Theorem \ref{thm:GRLs} for the family $\{V_i\}_{i\in I}$, and hence
\begin{equation} \label{eq:erada-34} \res_{\prod_{i\in I} V_i}^{V_I} (gdh) = \sum_{i\in I} \res_{V_i}^{V_I} (gdh) = \sum_{i\in I} \res_{V_i}^{V} (gdh)\, .\end{equation}

    Let us now fix two families of subspaces of $V$, $\{V_i\}_{i\in I}$ and $\{W_i\}_{i\in
I}$, such that $V_i \subseteq W_i$ for all $i\in I$ and $g\in \ed_k (W_i)$ for every $g\in K$.

    Let us now write $A_{I} = \prod_{i\in I} V_i$,
$W_{I} =  \prod_{i\in I} W_j$. It is clear that $K$ is a subgroup $E
(W_I, A_I)$ by means of the diagonal embedding.

    If we replace $V_I$ with $W_I$ in formula (\ref{eq:erada-34}), from Theorem \ref{thm:GRLs} we also obtain the following result:

\begin{cor} \label{th:res-gen-rec}  Let $\{V_i\}_{i\in I}$ and $\{W_i\}_{i\in I}$
be two families of subspaces of $V$ such that $V_i \subseteq W_i$
for all $i\in I$ and $g \in \ed_k (W_i)$ for every $g \in K$. If $g,h \in K$, then:
$$\res^{W_I}_{A_I} (gdh) = \sum_{i\in I} \res^{W_i}_{V_i} (gdh)\,
,$$\noindent where only a finite number of terms are different
from 0.

Accordingly, if $\res^{W_I}_{A_I} = 0$, then:
$$\sum_{i\in I} \res^{W_i}_{V_i} (gdh) = 0\,
,$$\noindent where only a finite number of terms are different
from 0.
\end{cor}

\begin{rem} Note that the statement of Corollary \ref{th:res-gen-rec} holds when $A_{I} = \oplus_{i\in I} V_i$,
$W_{I} =  \oplus_{i\in I} W_j$, and $\res^{W_I}_{A_I} (gdh) = 0$ for $g,h \in K$.
\end{rem}

    Let us now consider a commutative group $G$ and a morphism of
groups \linebreak $\varphi\colon k\times \to G$.

If $X$ is again the above family of $k$-subspaces of $V_I$, from the X-symbol
$$\begin{aligned} f_{g,h}^{\varphi}\colon X &\longrightarrow G \\ A &\longmapsto \varphi [\res_{A,\varphi}^{V_I} (gdh)]\, ,\end{aligned}$$\noindent where $g,h \in K$, with the same arguments of Corollary \ref{th:res-gen-rec} one has that:

\begin{cor} \label{cor:res-gen-rec-vector} Let $\{V_i\}_{i\in I}$ and $\{W_i\}_{i\in I}$
be two families of subspaces of $V$ such that $V_i \subseteq W_i$ and $g\in \ed_k (W_i)$
for every $g\in K$. If $G$ is a commutative group and $\varphi\colon k
\to G$ is a morphism of groups, then:
$$ \varphi (\res^{W_I}_{A_I} (gdh)) = \sum_{i\in I} \varphi [\res^{W_i}_{V_i} (gdh)]\,
,$$\noindent where $g,h \in K$, and only a finite number of terms are different
from 0.

Accordingly, if $\res^{W_I}_{A_I} (gdh) = 0$, then:
$$\sum_{i\in I} \varphi [\res^{W_i}_{V_i} (gdh)] = 0 \in G\,
,$$\noindent where only a finite number of terms are different
from 0.
\end{cor}

\begin{rem} Similar to the above, it is clear that the statement of Corollary \ref{cor:res-gen-rec-vector} holds when $A_{I} = \oplus_{i\in I} V_i$,
$W_{I} =  \oplus_{i\in I} W_j$, and $\res^{W_I}_{A_I} (gdh) = 0$, for $g,h \in K$.
\end{rem}

\subsubsection{Residue Theorem for reciprocity-admissible families of subspaces.}\label{sss:recip-residue}

Recall from Subsection \ref{sss:recip-tame} that a ``reciprocity-admissible family'' is a set of subspaces $\{V_i\}_{i\in I}$ of $V$ satisfying
the following properties:
\begin{itemize}
\item $\sum_{i\in I} V_i \sim V$ \item $V_{i}  \cap [\underset
{j\ne i} \sum V_j] \sim \{0\}$ for all ${i\in I}$.
\end{itemize}

    If we consider again the family of $k$-subspaces of $V$
$${\tilde X} = \{\{0\}, V, V_i, V_{i_1} + \dots + V_{i_r}, V_{k_1}\cap \dots \cap V_{k_s}, \sum_{j\in J} V_j\}_{\{i,i_n,k_m\} \in I; J\subseteq I}\, ,$$\noindent the ${\tilde X}$-symbol
$$\begin{aligned} {\tilde f}_{g,h}\colon {\tilde X} &\longrightarrow k \\ A &\longmapsto \res_{A}^{V} (gdh)\, ,\end{aligned}$$\noindent with $g,h \in K$,
satisfies the hypothesis of Theorem \ref{thm:GRLs} with $B_J = \sum_{j\in I-J} V_j$ for the family $\{V_i\}_{i\in I}$.

Hence, since $\sum_{i\in I} V_i \sim V$, then $<\sigma, \tau>_{\sum_{i\in I} V_i}^{V} = 1$ -see expression (\ref{eprop:rqerq})-, and we obtain the following result:

\begin{cor} \label{th:res-rec} If $\{V_i\}_{i\in I}$ is a reciprocity-admissible family of
subspaces of $V$, for all $g,h \in K$ one has that:
$$\sum_{i\in I} \res^{V}_{V_i} (gdh) = 0\,
,$$\noindent where only a finite number of terms are different
from 0.
\end{cor}

And, if $\varphi\colon k \longrightarrow G$ is again a morphism of groups, we have that

\begin{cor} \label{th:res-genio-rec} If $\{V_i\}_{i\in I}$ is a reciprocity-admissible family of
subspaces of $V$, for all $g,h \in K$ one has that:
$$\sum_{i\in I} \varphi [\res^{V}_{V_i} (gdh)] = 0 \in G\,
,$$\noindent where only a finite number of terms are different
from 0.
\end{cor}

\subsubsection{Residue Theorem for a Complete Curve.}\label{sss:residue-theorem-proof}

 Let $C$ again be a
non-singular and irreducible curve over a perfect field $k$, and
let $\Sigma_C$ be its function field. If $x\in C$ is a closed
point, we set $A_x = {\widehat {\O_x}}$ (the completion of the
local ring $\O_x$), and $K_x = (\widehat {\O_x})_0$ (the field of
fractions of ${\widehat {\O_x}}$, which coincides with the
completion of ${\Sigma}_C$ with respect to the valuation ring
$\O_x$).

    If we set $$V = \underset
{x\in C} {{\prod}'} K_x = \{ f = (f_x) \text { such that } f_x \in
K_x \text { and } f_x\in A_x \text { for almost all } x\}\,
,$$\noindent then $\Sigma_C$ can be regarded as a subspace of $V$ by
means of the diagonal embedding.

    It is clear that $K = \Sigma_C$ is a subgroup of $E(V,\{A_x\}_{x\in C})$.

    Let us again consider the family of $k$-subspaces of $V$
consisting of $V_1 = \prod_{x\in C} A_x$ and $V_2 =
\Sigma_C$. Since $C$ is a complete curve, then $\{V_1,V_2\}$ is a
reciprocity-admissible family of $V$.

    Thus, bearing in mind that $\Sigma_C \subset E
(V,\{V_1,V_2\})$, and that $fV_2 = V_2$ or $f V_2 = \{0\}$ for all
$f\in \Sigma_C$, it follows from expression (\ref{eqprop:res-equal}) and Corollary \ref{th:res-rec} that
$$\res_{V_1}^V (gdh) = 0\, ,$$\noindent for all $g,h \in
\Sigma_C$.

    Moreover, since $V = V_1 + \underset {x\in C} \oplus
K_x$, $f[\underset {x\in C} \oplus K_x] = \underset {x\in C} \oplus K_x$ or $f[\underset {x\in C} \oplus K_x] = \{0\}$  for all
$f\in \Sigma_C$, and $V_1 \cap \underset {x\in C}
\oplus K_x = \underset {x\in C} \oplus A_x$, then it follows
from Lemma \ref{lem:res-izq-mon} that
$$\res_{V_1}^V (gdh) = \res_{\underset {x\in C}
\oplus A_x}^{\underset {x\in C} \oplus K_x} (gdh) = 0\, ,$$\noindent for all $g,h \in
\Sigma_C$.

    Thus, again setting:
\begin{itemize}
\item $V_x = \dots \oplus \{0\} \oplus A_x \oplus \{0\} \oplus
\dots$, \item $W_x = \dots \oplus \{0\} \oplus K_x \oplus \{0\}
\oplus \dots$,
\end{itemize}
the families $\{V_x\}_{x\in C}$ and $\{W_x\}_{x\in C}$ of
$k$-subspaces of $V$ satisfy the hypothesis of Corollary
\ref{th:res-gen-rec}, and since
$$\res_{A_x}^{K_x} (gdh) = \tr_{k(x)/k} [\res_x(gdh)]\, ,$$\noindent
for all $g,h \in \Sigma_C$, where $k(x)$ is the residue
class field of the closed point $x$, and $\tr_{k(x)/k}$ is the trace of the extension
$k\hookrightarrow k(x)$, then we have that
$$\sum_{x\in C} \tr_{k(x)/k} [\res_x(gdh)] = 0 \in
k\, ,$$\noindent which is the classical expression of the Residue Theorem for a complete curve over a perfect field.

\subsubsection{Reciprocity Law for the Segal-Wilson pairing.}\label{ss:Segal-Wilson-reciprocity}

    Similar to Subsection \ref{sss:arbitrary-Segal-Wilson}, let $k$ be a field of characteristic zero, and let $K$ be a commutative $k$-algebra.

    If $V$ is a $K$-module, let $k((z))$ be the field of Laurent series.

    Let $\{V_i\}_{i\in I}$ again be a family of $k$-subspaces of
$V$, and we consider again the subgroup of
${\ed}_k (V)$ defined by:
$$E (V, \{V_i\}_{i\in I}) = \{g \in {\ed}_k (V) \text{
such that } g (V_i) < V_i \text { for all i, }$$\noindent
$$\text{ and } g(V_i) = V_i \text{ or } g(V_i) = \{0\} \text { for almost all i}\}\,
.$$

    Let us assume that $K$ is a subgroup of $E (V, \{V_i\}_{i\in I})$.

    For each subspace $V_i$, and for all $g,h\in K$, the function:
 \begin{align*}
c_{V_i}^V\colon K\times K&\to k((z))^\times \\
(g,h)&\mapsto c_{V_i}^V(g,h):=\exp_{z^{2}}\big(\frac{1}{2}\res_{V_i}^{V}(gdh)\big) \, ,
\end{align*}
is the $2$-cocycle of $K$ with coefficients in $k((z))^\times$ defined in Subsection \ref{sss:arbitrary-Segal-Wilson}.

    Let us write again $A_{I} = \prod_{i\in I} V_i$,
$W_{I} =  \prod_{i\in I} W_j$. It is clear that $K$ is a subgroup $E
(W_I, A_I)$ by means of the diagonal embedding.

    Thus, it follows from the morphism of groups $$\begin{aligned} k &\longrightarrow k((z)))^\times \\
    x &\longmapsto \exp_{z^{2}}\big(\frac{1}{2} x\big)\, , \end{aligned}$$\noindent and from Corollary \ref{cor:res-gen-rec-vector} that:

\begin{cor} \label{cor:segal-gen-rec-vector} Let $\{V_i\}_{i\in I}$ and $\{W_i\}_{i\in I}$
be two families of subspaces of $V$ such that $V_i \subseteq W_i$ and $g\in \ed_k (W_i)$
for every $g\in K$. We thus have:
$$ c^{W_I}_{A_I} (g,h) = \prod_{i\in I} c^{W_i}_{V_i} (g,h)\,
,$$\noindent where $g,h \in K$, and only a finite number of terms are different
from 1.

Accordingly, if $c^{W_I}_{A_I} (g,h) = 1$, then:
$$\prod_{i\in I} c^{W_i}_{V_i} (g,h) = 1 \in k((z)))^\times\,
,$$\noindent where only a finite number of terms are different
from 1.
\end{cor}

\begin{rem} Similar to the above, it is clear that the statement of Corollary \ref{cor:segal-gen-rec-vector} holds when $A_{I} = \oplus_{i\in I} V_i$,
$W_{I} =  \oplus_{i\in I} W_j$, and $c^{W_I}_{A_I} (g,h) = 1$ for $g,h \in K$.
\end{rem}

Moreover, from Corollary \ref{th:res-genio-rec} one deduces that

\begin{cor} \label{th:segal-res-genio-rec} If $\{V_i\}_{i\in I}$ is a reciprocity-admissible family of
subspaces of $V$, for all $g,h \in K$ one has that:
$$\prod_{i\in I} c_{V_i} (g,h) = 1 \in k((z))^\times\,
,$$\noindent where only a finite number of terms are different
from 1.
\end{cor}

Finally, if $C$ is a
non-singular and irreducible curve over a field $k$ of characteristic zero, and $\Sigma_C$ is its function field, for each closed
point $x\in C$, we again set $A_x = {\widehat {\O_x}}$ and $K_x = (\widehat {\O_x})_0$.

    For each $g,h \in \Sigma_C$, we set $(g,h)^{SW}_x = c_{A_x}^{K_x} (g,h)\in k((z))^\times$. When $C$ is complete, with similar arguments to above it follows from Corollary \ref{cor:segal-gen-rec-vector}
and Corollary \ref{th:segal-res-genio-rec} that $$\prod_{x\in C} (g,h)^{SW}_x = 1\, ,$$\noindent which is the reciprocity law for the Segal-Wilson pairing offered in \cite{HP}.

\begin{rem}[Final Consideration] It should be noted that the General Reciprocity Law (Theorem \ref{thm:GRLs}) is valid for each vector space over an arbitrary
field and, therefore, it should be possible to study other
reciprocity laws of symbols or to deduce new explicit expressions
in Algebraic Number Theory using this method.
\end{rem}

\end{document}